\documentclass{IJCAS}

\usepackage{url}
\usepackage{array,tabularx}
\usepackage{multicol,empheq,verbatim} 
\newcommand{\norm}[1]{\left\lVert#1\right\rVert}
\usepackage{amsmath,amssymb,amsfonts}
\usepackage{graphicx}
\usepackage[ruled]{algorithm2e}
\usepackage{algorithmic}
\allowdisplaybreaks

\journalvolumn{23}
\journalnumber{2}
\journalyear{2025}
\setarticlestartpagenumber{638}

\allowdisplaybreaks

\begin{document}

\title{A Generalized Primal-Dual Correction Method for Saddle-Point Problems with a Nonlinear Coupling Operator}

\author{Sai Wang\orcid{0009-0003-5599-5033} and Yi Gong*\orcid{0000-0001-7392-8991}}

\begin{abstract}
The saddle-point problems (SPPs) with nonlinear coupling operators frequently arise in various control systems, such as dynamic programming optimization, H-infinity control, and Lyapunov stability analysis. However, traditional primal-dual methods are constrained by fixed regularization factors. In this paper, a novel generalized primal-dual correction method (GPD-CM) is proposed to adjust the values of regularization factors dynamically. It turns out that this method can achieve the minimum theoretical lower bound of regularization factors, allowing for larger step sizes under the convergence condition being satisfied. The convergence of the GPD-CM is directly achieved through a unified variational framework. Theoretical analysis shows that the proposed method can achieve an ergodic convergence rate of $O(1/t)$. Numerical results support our theoretical analysis for an SPP with an exponential coupling operator.
\end{abstract}

\begin{keywords}
Saddle-point problem, prediction-correction method, nonlinear optimization, variational analysis.
\end{keywords}

\maketitle

\makeAuthorInformation{
Manuscript received May 30, 2024; revised July 25, 2024; accepted July 29, 2024. Recommended by Guest Editor Jong Min Lee. This work was supported in part by the National Natural Science Foundation of China under Grant 62371218, and in part by Shenzhen Science and Technology Program under Grants KCXFZ20211020174802004 and JCYJ20241202125328038.\\

Sai~Wang~and~Yi~Gong are with the Department of Electrical
and Electronic Engineering, Southern University of Science and Technology,
1088 Xueyuan Avenue, Shenzhen 518055, China (e-mail: \{wangs8, gongy\}@sustech.edu.cn).\\
* Corresponding author.
}

\runningtitle{2024}{Sai Wang and Yi Gong}{A Generalized Primal-Dual Correction Method for Saddle-Point Problems with a Nonlinear Coupling Operator}{024}{0453}{8}

\section{INTRODUCTION}
Saddle-point problems (SPPs) are widely encountered in control engineering, appearing in areas such as Lyapunov stability analysis \cite{MM1}, dynamic network design \cite{MM2}, and optimal control design. Additionally, many convex control problems with nonlinear constraints—like model predictive control, support vector machines, and H-infinity control—can often be reformulated as SPPs involving nonlinear coupling operators \cite{MM3,JC1,JC2}. Despite their prevalence, solving these SPPs with nonlinear coupling operators presents significant challenges. In this paper, we consider an SPP with a nonlinear coupling operator:
\begin{equation}
\min_{\mathbf{v}\in \mathcal{V}}\max_{\mathbf{w}\in \mathcal{W}}\mathcal{L}(\mathbf{v},\mathbf{w}):=f(\mathbf{v})+\langle\mathbf{w},\boldsymbol{\Phi}(\mathbf{v})\rangle-g(\mathbf{w}),\label{e1}
\end{equation}

\noindent where two closed sets $\mathcal{V}\subseteq \mathbb{R}^{n}$ and $\mathcal{W}\subseteq \mathbb{R}^{m}$ are convex. Two convex functions $\{f:\mathcal{V}\rightarrow \mathbb{R}$ and $g:\mathcal{W}\rightarrow \mathbb{R}\}$ are not necessarily smooth. The nonlinear function $\boldsymbol{\Phi}:\mathbb{R}^{n}\rightarrow \mathbb{R}^{m}$ is both convex and continuously differentiable over $\mathcal{V}$. As a result, the nonlinear coupling operator $\langle\mathbf{w},\boldsymbol{\Phi}(\mathbf{v})\rangle$ is convex on $\mathcal{V}$ and linear on $\mathcal{W}$. For a given SPP, its saddle point is depicted in Fig. \ref{fig:01}. To get the saddle point of (\ref{e1}), for the $k$-th iteration, the ideal proximal point algorithm generates a new sequence by solving the following problem:
\begin{equation}
\begin{aligned}
(\mathbf{v}^{k+1},\mathbf{w}^{k+1})=\arg\min\max\left\{\mathcal{L}(\mathbf{v},\mathbf{w})+\frac{r_{k}}{2}\Vert \mathbf{v}-\mathbf{v}^{k}\Vert^{2}\right.\\
\left.-\frac{s_{k}}{2}\Vert\mathbf{w}-\mathbf{w}^{k} \Vert^{2}\mid \mathbf{v}\in \mathcal{V}, \mathbf{w}\in \mathcal{W}\right\},    
\end{aligned}\label{vv}
\end{equation}
where $r_{k}>0, s_{k}>0$ are two regularization factors. The default norm used is the 2-norm. Updating both variables simultaneously presents a challenge, leading to the common practice of optimizing them alternately. 
\begin{figure}[t]
    \centering
    \includegraphics[width=3.0in]{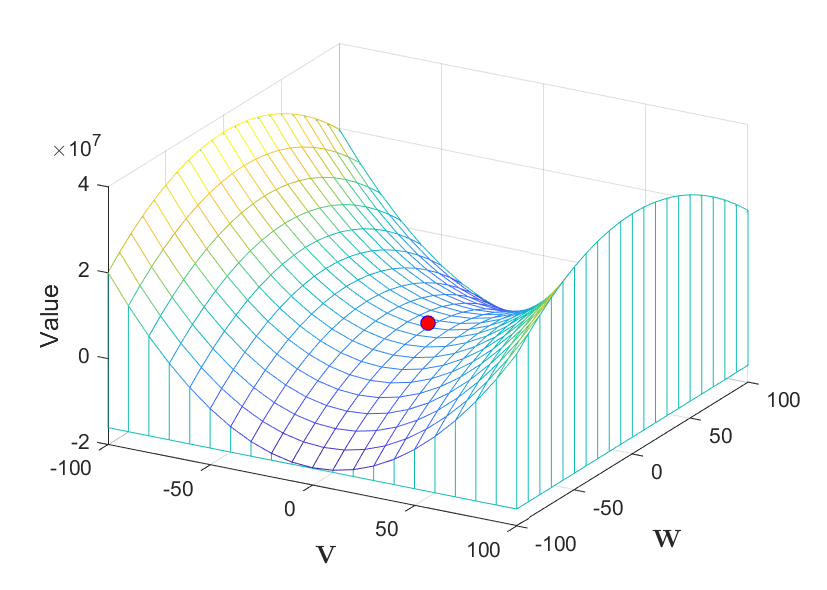}
    \caption{Saddle-point in a saddle-point problem.}
    \label{fig:01}
\end{figure}
In \cite{L1}, the Arrow-Hurwicz method was proposed for linear and nonlinear programming. To tackle SPPs with linear coupling operators, a variant known as the Chambolle and Pock method was presented in \cite{R4}. This method includes additional compensation steps and enforces stricter conditions to ensure convergence. 
 In \cite{New2}, the authors introduced a systematic procedure for analyzing degenerate preconditioned proximal point algorithms, enhancing splitting algorithms in convex minimization problems. In \cite{R9}, the authors proposed a novel accelerated primal-dual method without smoothing the objective function. Another advancement, as proposed in \cite{R10}, is an adaptive primal-dual method capable of automatically tuning step size parameters. Furthermore, in \cite{R5}, a generalized primal-dual method was developed to allow larger step sizes in resulting subproblems, offering a straightforward and versatile approach to enhancing numerical performance. Unfortunately, these developments primarily target SPPs with linear coupling operators. Recent developments have addressed the more intricate challenge of nonlinear coupling operators. Smooth convex-concave SPPs have been studied in \cite{R11}, exploring the iteration complexities of optimistic gradient and extra-gradient methods. Accelerated methods for composite structures were proposed in \cite{R15}, and \cite{R14} leveraged the alternating direction method of multipliers for decomposable SPPs. To address the problem involving the nonlinear coupling operator, the variant primal-dual hybrid gradient method (PDHGM) was developed in \cite{R3}. However, these primal-dual methods often rely on fixed regularization factors, which significantly hinder convergence rates. 

As noted in \cite{R5}, the numerical performance of primal-dual methods depends on the choice of regularization factors. A small theoretical lower bound of $r_{k}\cdot s_{k}$ allows for a larger step size. Inspired by insights from \cite{R5} and \cite{R6}, the generalized primal-dual correction method (GPD-CM) is proposed to enhance the step size by achieving a small theoretical lower bound. This approach offers potential benefits such as adaptive step sizes, trade-off control, and theoretical insights.
The contributions of this paper can be summarized as follows.
\begin{enumerate}
\item The  GPD-CM is developed to solve the SPP with the nonlinear coupling operator.
\item Under the variational framework, the convergence of the proposed method is established.
\item Compared to existing work, the proposed method achieves a larger step size by dynamically updating regularization parameters.
\end{enumerate}
This paper is structured as follows. Section \ref{section2} introduces preliminaries. In Section \ref{section3}, we design the GPD-CM and analyze its convergence. In Section \ref{section4}, numerical results are reported, and Section \ref{section5} concludes this paper.

\section{Preliminaries}\label{section2}
In this section, we introduce the preliminaries necessary for further analysis. The derivative of a vector function (a vector whose components are functions) $\boldsymbol{\Phi}(\mathbf{v})=[\phi_{1}(\mathbf{v}),\cdots,\phi_{m}(\mathbf{v})]^{T}$, concerning an input vector, $\mathbf{v}=[v_{1},\cdots,v_{n}]^{T}$, can be written as
\begin{equation}
\nabla  \boldsymbol{\Phi}(\mathbf{v})=\left( \begin{array}{c} \nabla \phi_{1}(\mathbf{v}) \\ \vdots\\ \nabla \phi_{m}(\mathbf{v})\end{array} \right)=\left( \begin{array}{ccc} \frac{\partial \phi_{1}}{\partial v_{1}} &\cdots&\frac{\partial \phi_{1}}{\partial v_{n}} \\ \vdots&\ddots&\vdots\\ \frac{\partial \phi_{m}}{\partial v_{1}}  &\cdots&\frac{\partial \phi_{m}}{\partial v_{n}} \end{array} \right).\nonumber
\end{equation}

\begin{lemma}\label{lemma1}
Let $h_{1}(\mathbf{x})$ and $h_{2}(\mathbf{x})$ be convex, and $h_{2}(\mathbf{x})$ be continuously differentiable. If the solution set of the minimization problem $\min\{h_{1}(\mathbf{x})+h_{2}(\mathbf{x})\vert \mathbf{x}\in \mathcal{X}\}$ is nonempty, for a closed convex set $\mathcal{X}\subseteq \mathbb{R}^{n}$, the vector $\mathbf{x}^{*}$ represents an optimal solutions, i.e., 
\begin{equation}
\mathbf{x}^{*}\in \arg \min \left\{h_{1}(\mathbf{x})+h_{2}(\mathbf{x})\mid\mathbf{x}\in \mathcal{X}\right\}. \label{old1}
\end{equation}
For $\mathbf{x}^{*}\in \mathcal{X}$, we can get the variational inequality
\begin{equation}
h_{1}(\mathbf{x})-h_{1}(\mathbf{x}^{*})+(\mathbf{x}-\mathbf{x}^{*})^{T}\nabla h_{2}(\mathbf{x}^{*})\ge0,\quad\forall \ \mathbf{x}\in \mathcal{X}.\label{old2}
\end{equation}
\end{lemma}
\begin{proof}The proof can be found in \cite{R5} (Lemma 2.1).\end{proof}
\noindent The saddle point of the SPP with the nonlinear coupling operator $(\mathbf{v}^{*},\mathbf{w}^{*})$ satisfies 
\begin{subequations}
\begin{empheq}[left=\empheqlbrace]{alignat=2}
         &\mathbf{v}^{*}\in\arg\min\left\{\mathcal{L}(\mathbf{v},\mathbf{w}^{*})\mid \mathbf{v}\in \mathcal{V}\right\},\label{ss}\\ 
         &\mathbf{w}^{*}\in\arg\max\left\{\mathcal{L}(\mathbf{v}^{*},\mathbf{w})\mid \mathbf{w}\in \mathcal{W}\right\}.\label{sss}
                \end{empheq}
\end{subequations}
The above optimal condition can be written as a saddle point inequality:
\begin{equation}
 \forall \mathbf{w}\in\mathcal{W}, \forall \mathbf{v}\in\mathcal{V},\mathcal{L}(\mathbf{v}^{*},\mathbf{w})\leq \mathcal{L}(\mathbf{v}^{*},\mathbf{w}^{*})\leq \mathcal{L}(\mathbf{v},\mathbf{w}^{*}).
\end{equation}
By Lemma \ref{lemma1}, for $\mathbf{v}^{*}\in \mathcal{V}$ and $\mathbf{w}^{*}\in  \mathcal{W}$, formulae (\ref{ss}) and (\ref{sss}) can be further expressed as the following variational inequalities:
\begin{empheq}[left=\empheqlbrace]{alignat=2}
& f(\mathbf{v})-f(\mathbf{v}^{*})+(\mathbf{v}-\mathbf{v}^{*})^{T}\nabla \boldsymbol{\Phi}(\mathbf{v}^{*})^{T}\mathbf{w}^{*}\ge0,& \  \forall \ \mathbf{v}\in \mathcal{V},\nonumber\\ 
 &g(\mathbf{w})-g(\mathbf{w}^{*})+(\mathbf{w}-\mathbf{w}^{*})^{T}[-\boldsymbol{\Phi}(\mathbf{v}^{*})]\ge0,&\ \forall \ \mathbf{w}\in \mathcal{W},\nonumber
 \end{empheq}
where the second variational inequality arises from a maximization problem, which is the reason for the minus sign appearing in the inequality. The above two inequalities can be reformulated as a monotone variational inequality:
\begin{equation}
  \vartheta(\boldsymbol{\eta})-\vartheta(\boldsymbol{\eta}^{*})+(\boldsymbol{\eta}-\boldsymbol{\eta}^{*})^{T}\boldsymbol{\Gamma}(\boldsymbol{\eta}^{*})\ge0, \quad\forall \ \boldsymbol{\eta}\in \mathcal{Z}, \label{e7}
\end{equation}
where $\vartheta(\boldsymbol{\eta})= f(\mathbf{v})+g(\mathbf{w})$,
\begin{equation}
 \boldsymbol{\eta}=\left( \begin{array}{c} \mathbf{v} \\[0.1cm]
           \mathbf{w} \end{array} \right), 
 \boldsymbol{\Gamma}(\boldsymbol{\eta})=\left( \begin{array}{c} \nabla \boldsymbol{\Phi}(\mathbf{v})^{T}\mathbf{w}\\ [0.1cm]-\boldsymbol{\Phi}(\mathbf{v})\end{array} \right)\mbox{ and } \mathcal{Z}=\mathcal{V}\times\mathcal{W}.\nonumber
\end{equation}
By Lemma 2, the coupling operator $\boldsymbol{\Gamma}(\boldsymbol{\eta})$ is monotonic. 

\begin{lemma}
Let two sets $\mathcal{V}\subseteq \mathbb{R}^{n}$ and $\mathcal{W}\subseteq \mathbb{R}^{m}$ be closed and convex.  For $\mathbf{v}\in \mathcal{V}$ and $\mathbf{w}\in \mathcal{W}$, the coupling operator $\langle \mathbf{w},\boldsymbol{\Phi}(\mathbf{v})\rangle$ is convex and continuously differentiable on $\mathcal{V}$ and linear on $\mathcal{W}$. 
The coupling operator $\boldsymbol{\Gamma}(\boldsymbol{\eta})$ satisfies the following monotonic form:
\begin{equation}
(\boldsymbol{\eta}-\tilde{\boldsymbol{\eta}})^{T}[\boldsymbol{\Gamma}(\boldsymbol{\eta})-\boldsymbol{\Gamma}(\tilde{\boldsymbol{\eta}})]\ge0, \quad\forall \ \boldsymbol{\eta}, \tilde{\boldsymbol{\eta}}\in \mathcal{Z}.\label{old20}
\end{equation}\label{lemma2}
\end{lemma}
\begin{proof}
Since $\langle \mathbf{w},\boldsymbol{\Phi}(\mathbf{v})\rangle$ is convex and continuously differentiable on $\mathcal{V}$, for any $\mathbf{v}\in \mathcal{V}$ and $\mathbf{w}\in \mathcal{W}$, we have 
\begin{equation}
\Theta(\mathbf{v},\mathbf{w}\mid \tilde{\mathbf{v}})=\mathbf{w}^{T}[\boldsymbol{\Phi}(\mathbf{v})-\boldsymbol{\Phi}(\tilde{\mathbf{v}}) - \nabla \boldsymbol{\Phi}(\tilde{\mathbf{v}})(\mathbf{v}-\tilde{\mathbf{v}})]\ge 0.\label{old22}
\end{equation}
Referring to (\ref{old22}), we obtain
\begin{equation}
\begin{aligned}
(\boldsymbol{\eta}-&\tilde{\boldsymbol{\eta}})^{T}[\boldsymbol{\Gamma}(\boldsymbol{\eta})-\boldsymbol{\Gamma}(\tilde{\boldsymbol{\eta}})]\\
=&\left( \begin{array}{c} \mathbf{v} -\tilde{\mathbf{v}}\\ \mathbf{w}-\tilde{\mathbf{w}} \end{array} \right)^{T}\left( \begin{array}{c} \nabla \boldsymbol{\Phi}(\mathbf{v})^{T}\mathbf{w} -\nabla \boldsymbol{\Phi}(\tilde{\mathbf{v}})^{T}\tilde{\mathbf{w}}\\
      -\boldsymbol{\Phi}(\mathbf{v}) +\boldsymbol{\Phi}(\tilde{\mathbf{v}})\end{array} \right)\\
 =&\ \mathbf{w}^{T}[\boldsymbol{\Phi}(\tilde{\mathbf{v}})-\boldsymbol{\Phi}(\mathbf{v})-\nabla \boldsymbol{\Phi}(\mathbf{v})(\tilde{\mathbf{v}}-\mathbf{v} ) ]\\
 &\ \ +\tilde{\mathbf{w}}^{T} [\boldsymbol{\Phi}(\mathbf{v}) -\boldsymbol{\Phi}(\tilde{\mathbf{v}})-\nabla \boldsymbol{\Phi}(\tilde{\mathbf{v}})(\mathbf{v} -\tilde{\mathbf{v}})]\\
 =&\ \Theta(\tilde{\mathbf{v}},\mathbf{w}\mid \mathbf{v})+\Theta(\mathbf{v},\tilde{\mathbf{w}}\mid \tilde{\mathbf{v}})\ge0.
\end{aligned}
\end{equation}
Thus, this lemma holds. 
\end{proof}

\section{Generalized Primal-Dual Correction Method}\label{section3}
In this section, we first introduce the classic Arrow-Hurwicz method. Following that, we present the GPD-CM. Lastly, we discuss the selection of the regularization parameter based on the convergence conditions.
\subsection{Arrow-Hurwicz Method }
For the $k$-th iteration, the Arrow-Hurwicz method introduced in \cite{R16} utilizes $({\mathbf{v}}^{k},{\mathbf{w}}^{k})$ to generate a pair of $(\tilde{\mathbf{v}}^{k},\tilde{\mathbf{w}}^{k})$ by solving the following problems:
\begin{empheq}[left=\empheqlbrace]{alignat=2}
&\tilde{\mathbf{v}}^{k}=\arg\min\left\{\mathcal{L}(\mathbf{v},\mathbf{w}^{k})+\frac{r_{k}}{2}\norm{\mathbf{v}-\mathbf{v}^{k}}^{2}\mid \mathbf{v}\in \mathcal{V}\right\},\nonumber\\
&\tilde{\mathbf{w}}^{k}=\arg\max\left\{\mathcal{L}(\tilde{\mathbf{v}}^{k},\mathbf{w})-\frac{s_{k}}{2}\Vert\mathbf{w}-\mathbf{w}^{k}\Vert^{2}\mid \mathbf{w}\in \mathcal{W}\right\}.\nonumber
\end{empheq}
Since there is no correction step, we set $({\mathbf{v}}^{k+1},{\mathbf{w}}^{k+1})=(\tilde{\mathbf{v}}^{k},\tilde{\mathbf{w}}^{k})$. The above problems can be formulated as two variational inequalities, i.e., 
\begin{empheq}[left=\empheqlbrace]{alignat=4}
   &  f(\mathbf{v})-f(\tilde{\mathbf{v}}^{k})+(\mathbf{v}-\tilde{\mathbf{v}}^{k})^{T}[\nabla \boldsymbol{\Phi}(\tilde{\mathbf{v}}^{k})^{T}\mathbf{w}^{k}\nonumber\\
      &\qquad\qquad\qquad\qquad+r_{k}(\tilde{\mathbf{v}}^{k}-\mathbf{v}^{k})]\ge0, \quad \forall \ \mathbf{v}\in \mathcal{V},\nonumber\\ 
& g(\mathbf{w})-g(\tilde{\mathbf{w}}^{k}) +(\mathbf{w}-\tilde{\mathbf{w}}^{k})^{T}[-\boldsymbol{\Phi}(\tilde{\mathbf{v}}^{k})\nonumber\\
  &\qquad\qquad\qquad\qquad+s_{k}(\tilde{\mathbf{w}}^{k}-\mathbf{w}^{k})]\ge0, \quad \forall \ \mathbf{w}\in \mathcal{W}.\nonumber
\end{empheq}
Further, for any $ \boldsymbol{\eta}\in \mathcal{Z}$,  the above variational inequality can be written as 
\begin{equation}
  \vartheta(\boldsymbol{\eta})-\vartheta( \tilde{\boldsymbol{\eta}}^{k})+(\boldsymbol{\eta}- \tilde{\boldsymbol{\eta}}^{k})^{T}\left\{\boldsymbol{\Gamma}( \tilde{\boldsymbol{\eta}}^{k})+{\mathbf{H}}_{k}(\tilde{\boldsymbol{\eta}}^{k}-\boldsymbol{\eta}^{k})\right\}\ge0,\nonumber
\end{equation}
where the proximal matrix $\mathbf{H}_{k}$ exhibits an upper triangular structure, as shown below:
\begin{equation} {\mathbf{H}}_{k}=\left(
\begin{array}{cc}
r_{k}\mathbf{I}_{n} & -\nabla \boldsymbol{\Phi}(\tilde{\mathbf{v}}^{k})^{T} \\[0.1cm]
\boldsymbol{0}&s_{k}\mathbf{I}_{m} 
\end{array}
\right).\nonumber
\end{equation}
The Arrow-Hurwicz method with a constant step size may diverge for some bilinear SPPs, e.g., $\min_{x_{1}\ge 0, x_{2}\ge 0}\max_{y}\{x_{1}+x_{2}-y(x_{1}+x_{2})-(-y)\}$, which explains why its convergence often requires additional assumptions \cite{R16}.

\subsection{Generalized Primal-Dual Correction Method}
The GPD-CM consists of two main steps: the prediction step and the correction step. In the prediction step, we adopt an iterative scheme of the primal-dual method to update predictive variables. Then the proximal matrix becomes a predictive matrix
\begin{equation} \mathbf{Q}_{k}=\left(
\begin{array}{cc}
r_{k}\mathbf{I}_{n} & -\nabla \boldsymbol{\Phi}(\tilde{\mathbf{v}}^{k})^{T} \\[0.1cm]
-\alpha\nabla \boldsymbol{\Phi}(\tilde{\mathbf{v}}^{k})&s_{k}\mathbf{I}_{m} 
\end{array}
\right),\nonumber
\end{equation}
where the generalized parameter $\alpha$ is $[0,1]$. To achieve this, the optimization problems become
\begin{equation}
\tilde{\mathbf{v}}^{k}=\arg\min\left\{\mathcal{L}(\mathbf{v},\mathbf{w}^{k})+\frac{r_{k}}{2}\norm{\mathbf{v}-\mathbf{v}^{k}}^{2}\mid \mathbf{v}\in \mathcal{V}\right\},\label{e18}
\end{equation}
\begin{equation}
\begin{aligned}
\tilde{\mathbf{w}}^{k}=\arg\max\left\{\mathcal{L}(\tilde{\mathbf{v}}^{k},\mathbf{w})+\alpha\mathbf{w}^{T}\nabla \boldsymbol{\Phi}(\tilde{\mathbf{v}}^{k})(\tilde{\mathbf{v}}^{k}-\mathbf{v}^{k})\right.\\ \left.-\frac{s_{k}}{2}\Vert\mathbf{w}-\mathbf{w}^{k}\Vert^{2}\mid \mathbf{w}\in \mathcal{W}\right\}.
\end{aligned}\label{e19}
\end{equation}
By solving (\ref{e18}) and (\ref{e19}), we obtain the predictive variable $\tilde{\boldsymbol{\eta}}^{k}$ that follows the variational inequality:
\begin{equation}
  \vartheta(\boldsymbol{\eta})-\vartheta( \tilde{\boldsymbol{\eta}}^{k})+(\boldsymbol{\eta}- \tilde{\boldsymbol{\eta}}^{k})^{T}\left\{\boldsymbol{\Gamma}( \tilde{\boldsymbol{\eta}}^{k})+{\mathbf{Q}}_{k}(\tilde{\boldsymbol{\eta}}^{k}-\boldsymbol{\eta}^{k})\right\}\ge0.
\label{uii13}\end{equation}
In the correction step, we use a corrective matrix $\mathbf{M}_{k}$ to correct the predictive variables by the following equation:
\begin{equation} \left(\begin{array}{c}
\mathbf{v}^{k+1}\\[0.1cm]
\mathbf{w}^{k+1}
\end{array}
\right)= \left(\begin{array}{c}
\mathbf{v}^{k}\\[0.1cm]
\mathbf{w}^{k}
\end{array}
\right)- \mathbf{M}_{k}\left(\begin{array}{c}
\mathbf{v}^{k}-\tilde{\mathbf{v}}^{k}\\[0.1cm]
\mathbf{w}^{k}-\tilde{\mathbf{w}}^{k}
\end{array}
\right),\label{old44}
\end{equation}
where the corrective matrix $\mathbf{M}_{k}$ is not unique, but it needs to satisfy the following lemma. When $\boldsymbol{\Phi}(\mathbf{v})$ is linear, the methods developed in \cite{R3} and \cite{R12} align with our proposed method by setting the customized parameter $\alpha$ to 1.

\begin{lemma}\label{lemma3333}
Let $\{\boldsymbol{\eta}^{k},\tilde{\boldsymbol{\eta}}^{k},\boldsymbol{\eta}^{k+1}\}$ be generated by the GPD-CM. For predictive matrix $\mathbf{Q}_{k}$, if we can find a corrective matrix $\mathbf{M}_{k}$ that meets the following convergence condition, i.e.,
\begin{equation}
\boldsymbol{\Sigma}_{k}=\mathbf{Q}_{k}\mathbf{M}_{k}^{-1}\succ0, 
\mathbf{G}_{k}=\mathbf{Q}_{k}^{T}+\mathbf{Q}_{k}-\mathbf{M}_{k}^{T}\boldsymbol{\Sigma}_{k}\mathbf{M}_{k}\succ0,\label{old45}
\end{equation}
 then we have 
\begin{align}
 \vartheta(\boldsymbol{\eta})&- \vartheta(\tilde{\boldsymbol{\eta}}^{k})+(\boldsymbol{\eta}-\tilde{\boldsymbol{\eta}}^{k})^{T}\boldsymbol{\Gamma}(\tilde{\boldsymbol{\eta}}^{k})\ge\frac{1}{2}\left(\norm{\boldsymbol{\eta}-\boldsymbol{\eta}^{k+1}}_{\boldsymbol{\Sigma}_{k}}^{2}\right.\nonumber\\&\left.-\norm{\boldsymbol{\eta}-\boldsymbol{\eta}^{k}}_{\boldsymbol{\Sigma}_{k}}^{2}\right) +\frac{1}{2}\norm{\tilde{\boldsymbol{\eta}}{}^{k}-\boldsymbol{\eta}^{k}}_{\mathbf{G}_{k}}^{2}, \quad\forall \ \boldsymbol{\eta}\in \mathcal{Z}. \label{old46}\end{align}
\end{lemma}
\begin{proof}
Referring to (\ref{old44}), we have
\begin{equation}\begin{aligned}
(\boldsymbol{\eta}-\tilde{\boldsymbol{\eta}}^{k})^{T}\mathbf{Q}_{k}(\boldsymbol{\eta}^{k}-\tilde{\boldsymbol{\eta}}^{k})&=(\boldsymbol{\eta}-\tilde{\boldsymbol{\eta}}^{k})^{T}\boldsymbol{\Sigma}_{k}\mathbf{M}_{k}(\boldsymbol{\eta}^{k}-\tilde{\boldsymbol{\eta}}^{k})\\&=(\boldsymbol{\eta}-\tilde{\boldsymbol{\eta}}^{k})^{T}\boldsymbol{\Sigma}_{k}(\boldsymbol{\eta}^{k}-\boldsymbol{\eta}^{k+1}).\nonumber
\end{aligned}\end{equation}
By using the following formula
\begin{equation}
\begin{aligned}
(\boldsymbol{a}-\boldsymbol{b})^{T}\boldsymbol{\Sigma}_{k}(\boldsymbol{c}-\mathbf{d})=&\frac{1}{2}\left( \Vert\boldsymbol{a}-\mathbf{d}\Vert_{\boldsymbol{\Sigma}_{k}}^{2}-\Vert\boldsymbol{a}-\boldsymbol{c}\Vert_{\boldsymbol{\Sigma}_{k}}^{2}\right)\\&+\frac{1}{2}\left( \Vert\boldsymbol{b}-\boldsymbol{c}\Vert_{\boldsymbol{\Sigma}_{k}}^{2}-\Vert\boldsymbol{b}-\mathbf{d}\Vert_{\boldsymbol{\Sigma}_{k}}^{2}\right),\nonumber
\end{aligned}\end{equation}
we obtain
\begin{equation}
\begin{aligned}
(\boldsymbol{\eta}&-\tilde{\boldsymbol{\eta}}^{k})^{T}\boldsymbol{\Sigma}_{k}(\boldsymbol{\eta}^{k}-\boldsymbol{\eta}^{k+1})\\&=\frac{1}{2}\left( \norm{\boldsymbol{\eta}-\boldsymbol{\eta}^{k+1}}_{\boldsymbol{\Sigma}_{k}}^{2}-\norm{\boldsymbol{\eta}-\boldsymbol{\eta}^{k}}_{\boldsymbol{\Sigma}_{k}}^{2}\right)\\
&\quad+\frac{1}{2}\left( \norm{\tilde{\boldsymbol{\eta}}{}^{k}-\boldsymbol{\eta}^{k}}_{\boldsymbol{\Sigma}_{k}}^{2}-\norm{\tilde{\boldsymbol{\eta}}{}^{k}-\boldsymbol{\eta}^{k+1}}_{\boldsymbol{\Sigma}_{k}}^{2}\right).\nonumber
\end{aligned}
\end{equation}
This leads us to the next step, where we further analyze
\begin{eqnarray}
\lefteqn{\norm{\tilde{\boldsymbol{\eta}}{}^{k}-\boldsymbol{\eta}^{k}}_{\boldsymbol{\Sigma}_{k}}^{2}-\norm{\tilde{\boldsymbol{\eta}}{}^{k}-\boldsymbol{\eta}^{k+1}}_{\boldsymbol{\Sigma}_{k}}^{2}}\nonumber\\
&=&\norm{\tilde{\boldsymbol{\eta}}{}^{k}-\boldsymbol{\eta}^{k}}_{\boldsymbol{\Sigma}_{k}}^{2}-\norm{(\tilde{\boldsymbol{\eta}}{}^{k}-\boldsymbol{\eta}^{k})-(\boldsymbol{\eta}^{k+1}-\boldsymbol{\eta}^{k})}_{\boldsymbol{\Sigma}_{k}}^{2}\nonumber\\
&=&\norm{\tilde{\boldsymbol{\eta}}{}^{k}-\boldsymbol{\eta}^{k}}_{\boldsymbol{\Sigma}_{k}}^{2}-\norm{(\tilde{\boldsymbol{\eta}}{}^{k}-\boldsymbol{\eta}^{k})-\mathbf{M}_{k}(\tilde{\boldsymbol{\eta}}{}^{k}-\boldsymbol{\eta}^{k})}_{\boldsymbol{\Sigma}_{k}}^{2}\nonumber\\
&=&(\tilde{\boldsymbol{\eta}}{}^{k}-\boldsymbol{\eta}^{k})^{T}(\mathbf{Q}_{k}^{T}+\mathbf{Q}_{k}-\mathbf{M}_{k}^{T}\boldsymbol{\Sigma}_{k}\mathbf{M}_{k})(\tilde{\boldsymbol{\eta}}{}^{k}-\boldsymbol{\eta}^{k})\nonumber\\
&=&\norm{\tilde{\boldsymbol{\eta}}{}^{k}-\boldsymbol{\eta}^{k}}_{\mathbf{G}_{k}}^{2}.\nonumber
\end{eqnarray}
Then we have
\begin{equation}
\begin{aligned}
(\boldsymbol{\eta}&-\tilde{\boldsymbol{\eta}}^{k})^{T}\mathbf{Q}_{k}(\boldsymbol{\eta}^{k}-\tilde{\boldsymbol{\eta}}^{k})=\frac{1}{2}\norm{\tilde{\boldsymbol{\eta}}{}^{k}-\boldsymbol{\eta}^{k}}_{\mathbf{G}_{k}}^{2}\\&+\frac{1}{2}\left( \norm{\boldsymbol{\eta}-\boldsymbol{\eta}^{k+1}}_{\boldsymbol{\Sigma}_{k}}^{2}-\norm{\boldsymbol{\eta}-\boldsymbol{\eta}^{k}}_{\boldsymbol{\Sigma}_{k}}^{2}\right).\label{eq4.6}
\end{aligned}
\end{equation}
Combining (\ref{uii13}) with (\ref{eq4.6}), this theorem is proved.
\end{proof}

\begin{theorem}\label{theorem1}
Let $\{\boldsymbol{\eta}^{k},\tilde{\boldsymbol{\eta}}^{k},\boldsymbol{\eta}^{k+1}\}$ be generated by the GPD-CM. For the predictive matrix $\mathbf{Q}_{k}$, if there is a corrective matrix $\mathbf{M}_{k}$ that satisfies the convergence condition (\ref{old45}), for any $\boldsymbol{\eta}^{*}\in \mathcal{Z}^{*}$, we obtain
\begin{equation}
\Vert\boldsymbol{\eta}^{*}-\boldsymbol{\eta}^{k+1}\Vert_{\boldsymbol{\Sigma}_{k}}^{2}\leq\Vert\boldsymbol{\eta}^{*}-\boldsymbol{\eta}^{k}\Vert_{\boldsymbol{\Sigma}_{k}}^{2}-\Vert\boldsymbol{\eta}^{k}-\tilde{\boldsymbol{\eta}}^{k}\Vert_{\mathbf{G}_{k}}^{2},  \label{old52}
\end{equation}
where $\mathcal{Z}^{*}$ is the set of optimal solutions.
\end{theorem}
\begin{proof} Setting $\boldsymbol{\eta}=\boldsymbol{\eta}^{*}$ in (\ref{old46}), we get 
\begin{equation}\begin{aligned}
&\norm{\boldsymbol{\eta}^{*}-\boldsymbol{\eta}^{k}}_{\boldsymbol{\Sigma}_{k}}^{2}-\norm{\boldsymbol{\eta}^{*}-\boldsymbol{\eta}^{k+1}}_{\boldsymbol{\Sigma}_{k}}^{2}-\Vert\boldsymbol{\eta}^{k}-\tilde{\boldsymbol{\eta}}^{k}\Vert_{\mathbf{G}_{k}}^{2}\\&\quad\ge 2 [\vartheta(\tilde{\boldsymbol{\eta}}^{k})- \vartheta(\boldsymbol{\eta}^{*})+(\tilde{\boldsymbol{\eta}}^{k}-\boldsymbol{\eta}^{*})^{T}\boldsymbol{\Gamma}(\tilde{\boldsymbol{\eta}}^{k})].\end{aligned}
\end{equation}
By Lemme \ref{lemma2}, the monotone operator satisfies $(\tilde{\boldsymbol{\eta}}^{k}-\boldsymbol{\eta}^{*})^{T}\boldsymbol{\Gamma}(\tilde{\boldsymbol{\eta}}^{k})\ge (\tilde{\boldsymbol{\eta}}^{k}-\boldsymbol{\eta}^{*})^{T}\boldsymbol{\Gamma}(\boldsymbol{\eta}^{*}).$ Then we have
\begin{equation}\begin{aligned}
&\Vert\boldsymbol{\eta}^{*}-\boldsymbol{\eta}^{k}\Vert_{\boldsymbol{\Sigma}_{k}}^{2}-\Vert\boldsymbol{\eta}^{*}-\boldsymbol{\eta}^{k+1}\Vert_{\boldsymbol{\Sigma}_{k}}^{2}-\Vert\boldsymbol{\eta}^{k}-\tilde{\boldsymbol{\eta}}^{k}\Vert_{\mathbf{G}_{k}}^{2}\\&\quad\ge  2[ \vartheta(\tilde{\boldsymbol{\eta}}^{k})- \vartheta(\boldsymbol{\eta}^{*})+(\tilde{\boldsymbol{\eta}}^{k}-\boldsymbol{\eta}^{*})^{T}\boldsymbol{\Gamma}(\boldsymbol{\eta}^{*})]\ge0.\end{aligned}
\end{equation}
Further, we obtain
\begin{equation}
\Vert\boldsymbol{\eta}^{*}-\boldsymbol{\eta}^{k}\Vert_{\boldsymbol{\Sigma}_{k}}^{2}-\Vert\boldsymbol{\eta}^{*}-\boldsymbol{\eta}^{k+1}\Vert_{\boldsymbol{\Sigma}_{k}}^{2}-\Vert\boldsymbol{\eta}^{k}-\tilde{\boldsymbol{\eta}}^{k}\Vert_{\mathbf{G}_{k}}^{2}\ge0.
\end{equation}
Thus, this theorem holds.
\end{proof}

\begin{figure*}[!h]
\begin{subequations}
\begin{align} 
&\boldsymbol{\Sigma}_{k}=\mathbf{Q}_{k}\mathbf{M}_{k}^{-1}=\left(\begin{array}{cc}
r_{k}\mathbf{I}_{n} & -\alpha\nabla \boldsymbol{\Phi}(\tilde{\mathbf{v}}^{k})^{T} \\[0.1cm]
-\alpha\nabla \boldsymbol{\Phi}(\tilde{\mathbf{v}}^{k})&\frac{\alpha(\alpha-1)}{r_{k}}\nabla \boldsymbol{\Phi}(\tilde{\mathbf{v}}^{k})\nabla \boldsymbol{\Phi}(\tilde{\mathbf{v}}^{k})^{T}+s_{k}\mathbf{I}_{m} 
\end{array}\right),\label{eq23a}\\
&\mathbf{G}_{k}=\mathbf{Q}_{k}^{T}+\mathbf{Q}_{k}-\mathbf{M}_{k}^{T}\boldsymbol{\Sigma}_{k}\mathbf{M}_{k}=\left(\begin{array}{cc}
r_{k}\mathbf{I}_{n} & -\alpha\nabla \boldsymbol{\Phi}(\tilde{\mathbf{v}}^{k})^{T} \\[0.1cm]
-\alpha\nabla \boldsymbol{\Phi}(\tilde{\mathbf{v}}^{k})&\frac{\alpha-1}{r_{k}}\nabla \boldsymbol{\Phi}(\tilde{\mathbf{v}}^{k})\nabla \boldsymbol{\Phi}(\tilde{\mathbf{v}}^{k})^{T}+s_{k}\mathbf{I}_{m} 
\end{array}\right).\label{eq23b}\end{align}
\end{subequations}
\hrulefill
\vspace{4pt}
\end{figure*}

\begin{remark}
In Lemma \ref{lemma3333}, the convergence condition is easily satisfied. In this paper, we select the corrective matrix as
\begin{equation}
\mathbf{M}_{k} = \left(
\begin{array}{cc}
\mathbf{I}_{n} & \frac{\alpha-1}{r_{k}} \nabla \boldsymbol{\Phi}(\tilde{\mathbf{v}}^{k})^{T} \\[0.1cm]
\boldsymbol{0} & \mathbf{I}_{m} 
\end{array}
\right).\nonumber
\end{equation}
This allows us to derive the matrices $\boldsymbol{\Sigma}_{k}$ and $\mathbf{G}_{k}$, as shown at the top of the next page. In each iteration, we adjust the value of $r_{k} \cdot s_{k}$ to ensure both matrices are positive definite.
\end{remark}
By the Schur complement lemma and the singular decomposition method, the matrix $\boldsymbol{\Sigma}_{k}$ in (\ref{eq23a}) is positive definite under the following condition:
\begin{equation} 
r_{k}s_{k}>\alpha\Vert\nabla \boldsymbol{\Phi}(\tilde{\mathbf{v}}^{k})\nabla \boldsymbol{\Phi}(\tilde{\mathbf{v}}^{k})^{T}\Vert.\label{u3u}
\end{equation}
In the same way, the matrix $\mathbf{G}_{k}$ in (\ref{eq23b}) is positive definite under the following condition:
\begin{equation} r_{k}s_{k}>[\alpha^{2}-\alpha+1]\Vert\nabla \boldsymbol{\Phi}(\tilde{\mathbf{v}}^{k})\nabla \boldsymbol{\Phi}(\tilde{\mathbf{v}}^{k})^{T}\Vert.\label{old25}
\end{equation}
For the detailed proof, please refer to Appendix A.
 It is obvious that $\boldsymbol{\Sigma}_{k}\succ0$ and $\mathbf{G}_{k}\succ0$ when condition (\ref{old25}) holds. The minimum theoretical lower bound of $r_{k}s_{k}$ is obtained at $\alpha=1/2$.
Since the predictive and corrective matrices vary with the sequence, for the $k$-th iteration, we adjust the values of $(r_{k},s_{k})$ using the following formulae:
 \begin{align}
r_{k}&=\frac{1}{\mu}\Vert\nabla \boldsymbol{\Phi}(\mathbf{v}^{k})\Vert,\label{old26}\\
s_{k}&=\frac{\alpha^{2}-\alpha+1}{\sigma r_{k}}\Vert\nabla \boldsymbol{\Phi}(\tilde{\mathbf{v}}^{k})\nabla \boldsymbol{\Phi}(\tilde{\mathbf{v}}^{k})^{T}\Vert,\label{old27}
\end{align}
where factor parameters $\mu>0$ and $\sigma\in(0, 1)$ are used to balance the objective and the proximal term. By (\ref{old26}) and (\ref{old27}), condition (\ref{old25}) holds.

\begin{lemma}
The convex solution set of the variational inequality (\ref{e7}) can be characterized as
\begin{equation}
\mathcal{Z}^{*}=\underset{\boldsymbol{\eta} \in \mathcal{Z}}{\bigcap}\{\tilde{\boldsymbol{\eta} }\in \mathcal{Z}: \vartheta(\boldsymbol{\eta} )-\vartheta(\tilde{\boldsymbol{\eta} })+(\boldsymbol{\eta} -\tilde{\boldsymbol{\eta} })^{T}\boldsymbol{\Gamma}(\boldsymbol{\eta} )\ge0\}.\label{eq3.14}
\end{equation}\label{lemma4}
\end{lemma}\vspace{-1cm}
\begin{proof} The proof can be found in \cite{Hp2} (Theorem 2.1). \end{proof}
By Lemma \ref{lemma2}, for any $ \boldsymbol{\eta} , \tilde{\boldsymbol{\eta} }\in \mathcal{Z}$, the monotone operator can be written as
\begin{equation}
(\boldsymbol{\eta} -\tilde{\boldsymbol{\eta} })^{T}\boldsymbol{\Gamma}(\boldsymbol{\eta} )\ge(\boldsymbol{\eta} -\tilde{\boldsymbol{\eta} })^{T}\boldsymbol{\Gamma}(\tilde{\boldsymbol{\eta} }).\label{eq3.15}
\end{equation}
Setting $\tilde{\boldsymbol{\eta} }=\boldsymbol{\eta} ^{*}$ in (\ref{eq3.15}), variational inequality (\ref{e7}) can be rewritten as the following equivalent form:
\begin{equation}
\boldsymbol{\eta} ^{*}\in \mathcal{Z}, \ \vartheta(\boldsymbol{\eta})- \vartheta(\boldsymbol{\eta}^{*})+(\boldsymbol{\eta} -\boldsymbol{\eta} ^{*})^{T}\boldsymbol{\Gamma}(\boldsymbol{\eta} )\ge0, \ \forall \ \boldsymbol{\eta} \in \mathcal{Z}. \label{eq3.16}
\end{equation}
Referring to the optimal condition (\ref{eq3.16}), for the given $\epsilon >0$, an approximate solution of variational inequality (\ref{e7}) satisfies 
\begin{equation}
 \tilde{\boldsymbol{\eta} }\in \mathcal{Z}, \vartheta(\boldsymbol{\eta})-\vartheta(\tilde{\boldsymbol{\eta}})+(\boldsymbol{\eta} -\tilde{\boldsymbol{\eta} })^{T}\boldsymbol{\Gamma}(\boldsymbol{\eta} )\ge-\epsilon, \ \forall \ \boldsymbol{\eta} \in \mathsf{N}_{\epsilon}(\tilde{\boldsymbol{\eta} }),\nonumber
\end{equation}
where $\mathsf{N}_{\epsilon}(\tilde{\boldsymbol{\eta} })=\{\boldsymbol{\eta} \mid \Vert\boldsymbol{\eta} -\tilde{\boldsymbol{\eta} }\Vert\leq \epsilon\}$ is the unit neighborhood of $\tilde{\boldsymbol{\eta} }$. For the given $\epsilon>0$, an iterative scheme can provide an approximate solution $\tilde{\boldsymbol{\eta} }\in \mathcal{Z}$, such that
\begin{equation}
\sup_{\boldsymbol{\eta} \in\mathsf{N}_{\epsilon}(\tilde{\boldsymbol{\eta} })} \{\vartheta(\tilde{\boldsymbol{\eta}})-\vartheta(\boldsymbol{\eta})+(\tilde{\boldsymbol{\eta} }-\boldsymbol{\eta} )^{T}\boldsymbol{\Gamma}(\boldsymbol{\eta} )\}\leq\epsilon. \label{eq3.17}
\end{equation}
Referring to the monotone operator $\boldsymbol{\Gamma}(\boldsymbol{\eta} )$, for $\forall \ \boldsymbol{\eta} \in \mathcal{Z}$, formula (\ref{old46}) can be expressed as
\begin{equation}
\begin{aligned}
\vartheta(\boldsymbol{\eta})&-\vartheta(\tilde{\boldsymbol{\eta}}{}^{k})+(\boldsymbol{\eta} -\tilde{\boldsymbol{\eta} }{}^{k})\boldsymbol{\Gamma}(\boldsymbol{\eta} )\\&+\frac{1}{2}\left(\Vert \boldsymbol{\eta} -\boldsymbol{\eta} ^{k} \Vert_{\boldsymbol{\Sigma}_{k}}^{2}-\Vert \boldsymbol{\eta} -\boldsymbol{\eta} ^{k+1} \Vert_{\boldsymbol{\Sigma}_{k}}^{2}\right)\ge0. \end{aligned}\label{eq3.18}
\end{equation}

\begin{assumption}
Let $\Omega\subset \mathbb{R}^{n}$ be a bounded convex set. The operator  $\nabla \Phi (\mathbf{x}):\mathbb{R}^{n}\to \mathbb{R}^{{m\times n}}$ is $L$-bounded and $L$-Lipschitz on $\Omega$, i.e., for all $\mathbf{x},\mathbf{y}\in \Omega$ it holds that 
\begin{align}
&\Vert \nabla \Phi (\mathbf{x})\Vert^{2}\leq L\norm{\mathbf{x}}^{2}\leq L C_{1}, && \nonumber\\
&\Vert \nabla \Phi (\mathbf{x})-\nabla \Phi (\mathbf{y})\Vert^{2}\leq L\norm{\mathbf{x}-\mathbf{y}}^{2}\leq LC_{2}, &&\nonumber
\end{align}
 where $L\in (0, \infty)$ is a positive constant, $C_{1}=\max_{\mathbf{x}\in \Omega} \norm{\mathbf{x}}^{2}$ and $C_{2}=\max_{\mathbf{x}-\mathbf{y}\in \Omega}\norm{\mathbf{x}-\mathbf{y}}^{2}$. 
\end{assumption}

\begin{lemma}
Suppose that the coupling function $\Phi$ satisfies Assumption 1 for some bounded convex set $\Omega\subset \mathbb{R}^{n}$, and that the iterates $\{\tilde{\mathbf{v}}^{k}\}_{k\in\mathbb{N}}$ generated by GPD-CM satisfy $\tilde{\mathbf{v}}^{k}\in \Omega$. Then, for any integer $t>0$, there exists a set of regularization factors $\{r_{k},s_{k}\}_{k=1}^{t}$ that satisfies $\{\boldsymbol{\Sigma}_{k-1}-\boldsymbol{\Sigma}_{k}\succ0, k=1,\cdots,t\}$, where $\boldsymbol{\Sigma}_{k}$ is defined in Lemma 3.
\label{lemma55}
\end{lemma}
\begin{proof}
The difference matrix $\boldsymbol{\Sigma}_{k-1}-\boldsymbol{\Sigma}_{k}$ is written as
\begin{equation}\begin{aligned}\left(
\begin{array}{cc}
(r_{k-1}-r_{k})\mathbf{I}_{n} & (\nabla \Phi (\tilde{\mathbf{v}}^{k})-\nabla \Phi (\tilde{\mathbf{v}}^{k-1}))^{T} \\[0.2cm]
\nabla \Phi (\tilde{\mathbf{v}}^{k})-\nabla \Phi (\tilde{\mathbf{v}}^{k-1})&(s_{k-1}-s_{k})\mathbf{I}_{m}
\end{array}
\right).\nonumber\end{aligned}
\end{equation}
Using a similar approach as deriving condition (\ref{u3u}), we can deduce the condition for the above symmetrical matrix to be positive semidefinite, which is
\begin{subequations}
\begin{empheq}[left={}\empheqlbrace]{alignat=3}
&0<r_{k}< r_{k-1}, 0<s_{k}< s_{k-1},\nonumber\\
&r_{k}s_{k}< r_{k-1}s_{k-1}< LC_{1},\nonumber\\
&(r_{k-1}-r_{k})(s_{k-1}-s_{k})> \Vert \nabla \Phi (\tilde{\mathbf{v}}^{k})-\nabla \Phi (\tilde{\mathbf{v}}^{k-1})\Vert^{2}.\nonumber
\end{empheq}
\end{subequations}
In fact, we need to design a monotonically decreasing sequence that fulfills the above condition. Since the values of $\{r_{k},s_{k}\}$ can be chosen arbitrarily from the interval $(0,\infty)$, such a sequence exists. Let $L$, $C_{1}$, and $C_{2}$ be defined as in Assumption 1. For example, we take
\begin{equation}
r_{t}=s_{t}=LC_{1}+\gamma,\nonumber
\end{equation}
where $\gamma>0$ and let $r_{t-1}=s_{t-1}$. Then we have
\begin{equation}
(r_{t-1}-r_{t})^{2}= LC_{2}+\gamma, \quad r_{t-1}=r_{t}+\sqrt{LC_{2}+\gamma}.\nonumber
\end{equation}
Then the sequence can be expressed by
\begin{equation}
r_{k}=(t-k)\sqrt{LC_{2}+\gamma}+LC_{1}+\gamma, k=0,1,\cdots,t.\nonumber
\end{equation}
Thus, this lemma is proved.
\end{proof}

\begin{theorem}
 Let $\{\tilde{\mathbf{v}}^{k}\}_{k\in\mathbb{N}}$ and $\{\tilde{\mathbf{w}}^{k}\}_{k\in \mathbb{N}}$ be two sequences generated by GPD-CM for solving the problem (1), both of which satisfy Assumption 1. Then we have $\tilde{\mathbf{v}}^{k},\tilde{\mathbf{w}}^{k}\in \Omega$. Define $\tilde{\mathbf{v}}_{t}$ and $\tilde{\boldsymbol{\eta} }_{t}$ for any positive integer $t$ as follows:
\begin{equation}
\tilde{\mathbf{v}}_{t}=\frac{1}{1+t}\sum_{k=0}^{t}\tilde{\mathbf{v}}^{k},\quad \tilde{\boldsymbol{\eta} }_{t}=\frac{1}{1+t}\sum_{k=0}^{t}\tilde{\boldsymbol{\eta} }^{k}.\nonumber
\end{equation}
For $\tilde{\boldsymbol{\eta} }_{t}\in \mathcal{Z}$ and $\forall \ \boldsymbol{\eta} \in \mathcal{Z}$, the GPD-CM converges ergodically with an $O(1/t)$ rate, 
\begin{equation}
 \vartheta(\tilde{\boldsymbol{\eta}}_{t})-\vartheta(\boldsymbol{\eta})+(\tilde{\boldsymbol{\eta} }_{t}-\boldsymbol{\eta} )^{T}\boldsymbol{\Gamma}(\boldsymbol{\eta} )\leq \frac{1}{2(1+t)}\Vert\boldsymbol{\eta} -\boldsymbol{\eta} ^{0}\Vert_{\boldsymbol{\Sigma}_{0}}^{2}. \nonumber
\end{equation}
\end{theorem}

\begin{proof}
For all $k=0, 1,\cdots,t$, $\tilde{\boldsymbol{\eta} }^{k}$ belongs to the convex set $\mathcal{Z}$. Then the linear combination $\tilde{\boldsymbol{\eta} }_{t}\in \mathcal{Z}$ holds.  Summing inequality (\ref{eq3.18}) over iterations $k=0,1,\cdots,t$, for any $ \boldsymbol{\eta} \in \mathcal{Z}$, the following inequality holds: 
\begin{equation}\begin{aligned}
(1+t)\vartheta(\boldsymbol{\eta})&-\sum_{k=0}^{t}\vartheta(\tilde{\boldsymbol{\eta}}{}^{k})+\left((1+t)\boldsymbol{\eta} -\sum_{k=0}^{t}\tilde{\boldsymbol{\eta} }{}^{k}\right)\boldsymbol{\Gamma}(\boldsymbol{\eta} )\\&+\frac{1}{2}\sum_{k=0}^{t}\left(\Vert \boldsymbol{\eta} -\boldsymbol{\eta} ^{k} \Vert_{\boldsymbol{\Sigma}_{k}}^{2}-\Vert \boldsymbol{\eta} -\boldsymbol{\eta} ^{k+1} \Vert_{\boldsymbol{\Sigma}_{k}}^{2}\right)\ge 0.\nonumber
\end{aligned}\end{equation}
The inequality above divided by $1+t$ equals
\begin{equation}
\begin{aligned}
\vartheta(\boldsymbol{\eta})&-\frac{1}{1+t}\sum_{k=0}^{t}\vartheta(\tilde{\boldsymbol{\eta}}{}^{k})+\left(\boldsymbol{\eta} -\frac{1}{1+t}\sum_{k=0}^{t}\tilde{\boldsymbol{\eta} }{}^{k}\right)\boldsymbol{\Gamma}(\boldsymbol{\eta} )\\&+\frac{1}{2 (1+t)}\sum_{k=0}^{t}\left(\Vert \boldsymbol{\eta} -\boldsymbol{\eta} ^{k} \Vert_{\boldsymbol{\Sigma}_{k}}^{2}-\Vert \boldsymbol{\eta} -\boldsymbol{\eta} ^{k+1} \Vert_{\boldsymbol{\Sigma}_{k}}^{2}\right)\ge 0.
\end{aligned}\nonumber
\end{equation}
We know that $\vartheta(\boldsymbol{\eta})$ is convex in $\mathcal{Z}$ and $\tilde{\boldsymbol{\eta}}_{t}=\frac{1}{1+t}\sum_{k=0}^{t}\tilde{\boldsymbol{\eta}}^{k}\in \mathcal{Z}$. So  $\vartheta(\tilde{\boldsymbol{\eta}}_{t})\leq\frac{1}{1+t}\sum_{k=0}^{t}\vartheta(\tilde{\boldsymbol{\eta}}^{k})$ holds. Then we have
\begin{equation}
\begin{aligned}
 &\vartheta(\boldsymbol{\eta})-\vartheta(\tilde{\boldsymbol{\eta}}_{t})+\left(\boldsymbol{\eta} -\tilde{\boldsymbol{\eta} }_{t}\right)\boldsymbol{\Gamma}(\boldsymbol{\eta} )\\&+\frac{1}{2 (1+t)}\sum_{k=0}^{t}\left(\Vert \boldsymbol{\eta} -\boldsymbol{\eta} ^{k} \Vert_{\boldsymbol{\Sigma}_{k}}^{2}-\Vert \boldsymbol{\eta} -\boldsymbol{\eta} ^{k+1} \Vert_{\boldsymbol{\Sigma}_{k}}^{2}\right)\ge 0, 
\\ &\vartheta(\boldsymbol{\eta})-\vartheta(\tilde{\boldsymbol{\eta}}_{t})+\left(\boldsymbol{\eta} -\tilde{\boldsymbol{\eta} }_{t}\right)\boldsymbol{\Gamma}(\boldsymbol{\eta} )\\&+\frac{1}{2 (1+t)}\left[\Vert \boldsymbol{\eta} -\boldsymbol{\eta} ^{0} \Vert_{\boldsymbol{\Sigma}_{0}}^{2}-\sum_{k=1}^{t}(\boldsymbol{\eta} -\boldsymbol{\eta} ^{k})^{T}\mathbf{D}_{k}(\boldsymbol{\eta} -\boldsymbol{\eta} ^{k})\right]\ge 0, \\
\end{aligned}\nonumber
\end{equation}
where $\mathbf{D}_{k}=\boldsymbol{\Sigma}_{k-1}-\boldsymbol{\Sigma}_{k}$. By Lemma \ref{lemma55}, we take $\mathbf{D}_{k}\succ0$. For any $\boldsymbol{\eta}\ne \boldsymbol{\eta}^{k}$, $\Vert\boldsymbol{\eta} -\boldsymbol{\eta} ^{k}\Vert^{2}_{\mathbf{D}_{k}}>0$ holds. We get
\begin{equation}
\begin{aligned}
&\vartheta(\boldsymbol{\eta})-\vartheta(\tilde{\boldsymbol{\eta}}_{t})+\left(\boldsymbol{\eta} -\tilde{\boldsymbol{\eta} }_{t}\right)\boldsymbol{\Gamma}(\boldsymbol{\eta} )+\frac{1}{2 (1+t)}\Vert \boldsymbol{\eta} -\boldsymbol{\eta} ^{0} \Vert_{\boldsymbol{\Sigma}_{0}}^{2}\\&\qquad\qquad\qquad\qquad\qquad\qquad\ge \frac{1}{2 (1+t)}\sum_{k=1}^{t}\Vert\boldsymbol{\eta} -\boldsymbol{\eta} ^{k}\Vert^{2}_{\mathbf{D}_{k}}, \\
&\vartheta(\boldsymbol{\eta})-\vartheta(\tilde{\boldsymbol{\eta}}_{t})+\left(\boldsymbol{\eta} -\tilde{\boldsymbol{\eta} }_{t}\right)\boldsymbol{\Gamma}(\boldsymbol{\eta} )+\frac{1}{2 (1+t)}\Vert \boldsymbol{\eta} -\boldsymbol{\eta} ^{0} \Vert_{\boldsymbol{\Sigma}_{0}}^{2}\ge 0.\\
\end{aligned}\nonumber
\end{equation}
Thus, this theorem holds.
\end{proof}
The pseudo-code of the GPD-CM is provided in Algorithm 1.

\begin{algorithm}[t]\label{algorithm1}
\caption{GPD-CM }
\LinesNumbered
\KwIn{ Parameters $\mu>1, \sigma\in (0,1), \alpha\in[0,1], \tau>0, ite \ge 1$.}
\KwOut{The optimal solution: $\mathbf{v}^{opt}$.}
Initialize $\mathbf{v}^{0},\mathbf{w}^{0},k=0$\;
\While{$err\ge\tau$ \& $k \le ite$}{
\emph{\% The prediction step:}\\
$r_{k}=\frac{1}{\mu}\Vert\nabla \boldsymbol{\Phi}(\mathbf{v}^{k})\Vert$\;
$\tilde{\mathbf{v}}^{k}=\arg\min\{\mathcal{L}(\mathbf{v},\mathbf{w}^{k})+\frac{r_{k}}{2}\norm{\mathbf{v}-\mathbf{v}^{k}}^{2}\mid \mathbf{v}\in \mathcal{V}\}$\;
$s_{k}=[\alpha+(1-\alpha)^{2}]\frac{1}{\sigma r_{k}}\Vert\nabla \boldsymbol{\Phi}(\tilde{\mathbf{v}}^{k})\nabla \boldsymbol{\Phi}(\tilde{\mathbf{v}}^{k})^{T}\Vert$\;
$\tilde{\mathbf{w}}^{k}=\arg\max\{\mathcal{L}(\tilde{\mathbf{v}}^{k},\mathbf{w})+\alpha\mathbf{w}^{T}\nabla \boldsymbol{\Phi}(\tilde{\mathbf{v}}^{k})(\tilde{\mathbf{v}}^{k}-\mathbf{v}^{k})-\frac{s_{k}}{2}\Vert\mathbf{w}-\mathbf{w}^{k}\Vert^{2}\mid \mathbf{w}\in \mathcal{W}\}$\;
\emph{\% The correction step:}\\
$\boldsymbol{\eta}^{k+1}=\boldsymbol{\eta}^{k}-\mathbf{M}_{k}(\boldsymbol{\eta}^{k}-\tilde{\boldsymbol{\eta}}^{k})$\;
$err=\frac{1}{2}(\Vert\mathbf{v}^{k+1}-\mathbf{v}^{*}\Vert+\Vert\mathbf{w}^{k+1}-\mathbf{w}^{*}\Vert)$\;
$k=k+1$\;
} 
\Return { $\mathbf{v}^{opt}=\mathbf{v}^{k}$}\;
 \end{algorithm}

\section{Numerical Experiments}\label{section4}
In this section, to evaluate the performance of the proposed GPD-CM, we constructed an SPP with an exponential coupling operator derived from chemical reaction control, battery charging control, and population dynamics control:
\begin{equation}
\min_{\mathbf{v}\in \mathbb{R}^{n}}\max_{\mathbf{w}\in \mathbb{R}^{m}_{+}}\mathcal{L}(\mathbf{v},\mathbf{w})=\Vert\mathbf{A}\mathbf{v}-\mathbf{a}\Vert^{2}+\langle\mathbf{w},e^{\mathbf{B}\mathbf{v}}-\mathbf{b}\rangle-\Vert\mathbf{C}\mathbf{w}-\mathbf{c}\Vert^{2},\nonumber
\end{equation}
where the dimensions of constants are listed as $\mathbf{A}\in \mathbb{R}^{n\times n}$, $\mathbf{B}\in \mathbb{R}^{m\times n}$, $\mathbf{C}\in \mathbb{R}^{m\times m}$, $\mathbf{a}\in \mathbb{R}^{n}$, $\mathbf{b}\in \mathbb{R}^{m}$ and $\mathbf{c}\in \mathbb{R}^{m}$. 
We use the GPD-CM to solve the above problem. In the prediction step, for the $k$-th iteration, we first generate $r_{k}$ by using (\ref{old26}) and then solve the following problem: 
\begin{equation}\tilde{\mathbf{v}}^{k}=\arg\min\{\mathcal{L}(\mathbf{v},\mathbf{w}^{k})+\frac{r_{k}}{2}\norm{\mathbf{v}-\mathbf{v}^{k}}^{2}\mid \mathbf{v}\in \mathbb{R}^{n}\}.\label{old29}\end{equation}
We can obtain it by the gradient descent method. Next, we update the values of $s_{k}$ and $\hat{\mathbf{w}}^{k}$ by using (\ref{old27}) and the following iterative scheme:
\begin{equation}\begin{aligned}\hat{\mathbf{w}}^{k}&=\arg\max\{\mathcal{L}(\tilde{\mathbf{v}}^{k},\mathbf{w})+\alpha\mathbf{w}^{T}\nabla \boldsymbol{\Phi}(\tilde{\mathbf{v}}^{k})(\tilde{\mathbf{v}}^{k}-\mathbf{v}^{k})\\
&\qquad \qquad \qquad \qquad\quad-\frac{s_{k}}{2}\Vert\mathbf{w}-\mathbf{w}^{k}\Vert^{2}\mid \mathbf{w}\in \mathbb{R}^{m}\}\\
&=(2\mathbf{C}^{T}\mathbf{C}+s_{k}\mathbf{I}_{m})^{-1}[\alpha\mbox{diag}(e^{\mathbf{B}\tilde{\mathbf{v}}^{k}})\mathbf{B}(\tilde{\mathbf{v}}^{k}-\mathbf{v}^{k})\\&\qquad\qquad\qquad\qquad\quad+e^{\mathbf{B}\mathbf{v}}-\mathbf{b}+2\mathbf{C}^{T}\mathbf{c}+s_{k}\mathbf{w}^{k}],\nonumber
\end{aligned}\end{equation}
where $\mbox{diag}(\cdot)$ denotes the diagonal matrix formed by a vector. We project the variable $\hat{\mathbf{w}}^{k}$ on the set $\mathbb{R}^{m}_{+}$, as shown below:
\begin{equation}\tilde{w}^{k}_{i}=\max\{\hat{w}^{k}_{i},0\},\ i=1,\cdots m.\nonumber\end{equation}
In the correction step, we correct the predictive variables:
\begin{equation} \left(\begin{array}{c}
\mathbf{v}^{k+1}\\[0.1cm]
\mathbf{w}^{k+1}
\end{array}
\right)= \left(\begin{array}{c}
\mathbf{v}^{k}\\[0.1cm]
\mathbf{w}^{k}
\end{array}
\right)- \mathbf{M}_{k}\left(\begin{array}{c}
\mathbf{v}^{k}-\tilde{\mathbf{v}}^{k}\\[0.1cm]
\mathbf{w}^{k}-\tilde{\mathbf{w}}^{k}
\end{array}
\right),\nonumber
\end{equation}
where the corrective matrix is given as
\begin{equation}
\mathbf{M}_{k}=\left(\begin{array}{cc}
\mathbf{I}_{n} & \frac{\alpha-1}{r_{k}} [\mbox{diag}(e^{\mathbf{B}\tilde{\mathbf{v}}^{k}})\mathbf{B}]^{T} \\[0.1cm]
\boldsymbol{0}&\mathbf{I}_{m} 
\end{array}
\right). \nonumber
\end{equation}

\begin{figure}[t]
    \centering
    \includegraphics[width=3.35in]{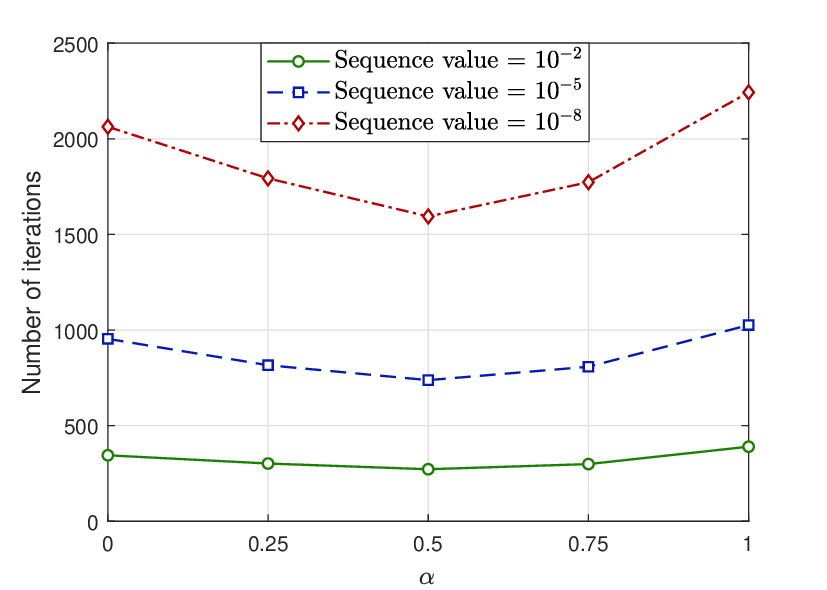}
    \caption{Different sequence values: Number of iterations versus values of $\alpha$.}
    \label{fig:1}
\end{figure}

\noindent In the experiment, the dimension of $(n,m)$ is set to $10\times10$, and all matrices and vectors are randomly generated in Matlab by using the command $\mathsf{rand(m,n)}$ while ensuring that $\mathcal{L}(\mathbf{v},\mathbf{w})$ is an SPP. The initial values of $(\mathbf{v},\mathbf{w})$ are set to $\mathbf{v}^{0}=2\cdot\mathsf{ones(n,1)}$ and $\mathbf{w}^{0}=\mathsf{ones(m,1)}$, respectively. The regularization parameters are set to $\mu=1, \sigma=0.95,$ and $\alpha=0.5$. In addition, all methods are implemented serially. The sequence value is defined as $\frac{1}{2}(\Vert\mathbf{v}^{k+1}-\mathbf{v}^{*}\Vert+\Vert\mathbf{w}^{k+1}-\mathbf{w}^{*}\Vert)$. All simulation results are averaged over 100 random experiments. The source code of this project can be found on GitHub at \url{https://github.com/SaiWang-Neo/GPD-CM-for-SPP}.

\begin{figure}[t]
    \centering
    \includegraphics[width=3.35in]{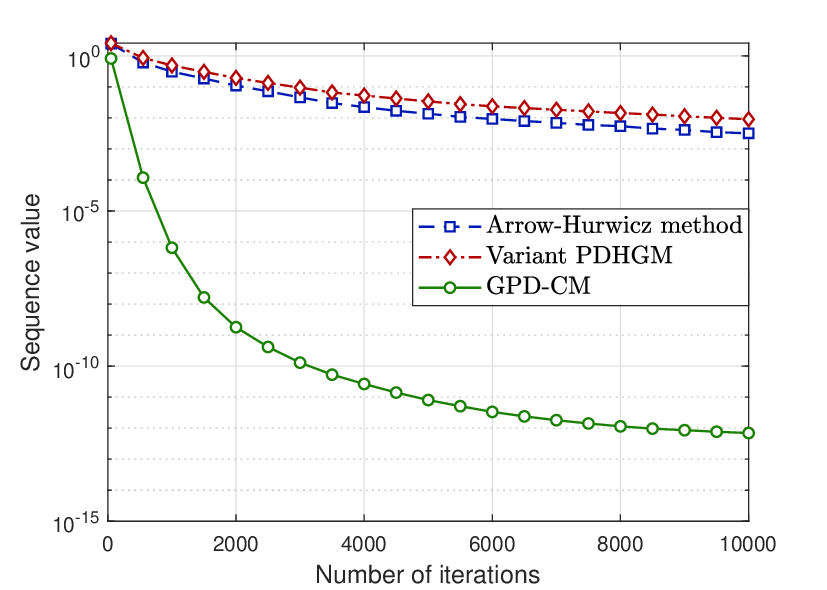}
    \caption{Comparsion of methods: Sequence value versus the number of iterations.}
    \label{fig:2}
\end{figure}

Fig. \ref{fig:1} illustrates the number of iterations versus different values of \( \alpha \) for different sequence values of the GPD-CM. It can be observed that the sequence value of the proposed method decreases as the number of iterations increases. Notably, under the convergence condition being satisfied, the method performs best when \( \alpha = 0.5 \), as it achieves the minimum lower bound of \( r_{k} \cdot s_{k} \), which equals \( \frac{0.75}{\sigma}\Vert\nabla \boldsymbol{\Phi}(\tilde{\mathbf{v}}^{k})\nabla \boldsymbol{\Phi}(\tilde{\mathbf{v}}^{k})^{T}\Vert \). The results confirm that a theoretical lower bound of \( r_{k} \cdot s_{k} \) indeed leads to an enlarged step size.

Fig. \ref{fig:2} depicts the sequence value versus the number of iterations for the Arrow-Hurwicz method \cite{R16}, variant PDHGM \cite{R3}, and GPD-CM. Although all methods converge, the proposed GPD-CM demonstrates a notably faster convergence rate compared to the others. This is attributed to the dynamic adjustment of regularization factors in our method, whereas the other methods rely on fixed regularization factors, often requiring larger values to ensure convergence, leading to slower convergence rates.

\section{Conclusion}\label{section5}

In this paper, we proposed a GPD-CM for solving SPPs featuring a nonlinear coupling operator. By leveraging a variational framework, we directly obtain the convergence of the proposed method. The proposed method has an ergodic convergence rate of $O(1/t)$. Numerical results show that our proposed method significantly outperforms the other methods in terms of the convergence rate.

\appendix
\label{app1}

 By the Schur complement lemma, we know that $\boldsymbol{\Sigma}_{k}\succ0$ if and only if $r_{k}s_{k}\mathbf{I}_{m}-\alpha\nabla \boldsymbol{\Phi}(\tilde{\mathbf{v}}^{k})\nabla \boldsymbol{\Phi}(\tilde{\mathbf{v}}^{k})^{T}\succ0$ holds. Let the singular value decomposition of $\nabla \Phi (\tilde{\mathbf{v}}{}^{k})$ be equal to $\mathbf{U}\boldsymbol{\Lambda}\mathbf{V}^{T}$, where $\mathbf{U}$ and $\mathbf{V}$ are unitary matrices and $\boldsymbol{\Lambda}$ is a rectangular diagonal matrix with the singular values lying on the diagonal.
\begin{equation}
\begin{aligned}
r_{k}s_{k}&\boldsymbol{I}_{m}-\alpha\nabla \Phi (\tilde{\mathbf{v}}{}^{k})\nabla \Phi (\tilde{\mathbf{v}}{}^{k})^{T}\\
=& r_{k}s_{k}\boldsymbol{I}_{m}-\alpha\mathbf{U}\boldsymbol{\Lambda}\mathbf{V}^{T}\mathbf{V}\boldsymbol{\Lambda}\mathbf{U}^{T}\\
=& r_{k}s_{k}\boldsymbol{I}_{m}-\alpha\mathbf{U}\boldsymbol{\Lambda}^{2}\mathbf{U}^{T}\\
=& \mathbf{U}\left(r_{k}s_{k}\boldsymbol{I}_{m}-\alpha\boldsymbol{\Lambda}^{2}\right)\mathbf{U}^{T}
\end{aligned}\label{eq30ee}
\end{equation}
For given $r_{k}>0, s_{k}>0$ and $\alpha\in[0, 1]$, the symmetric matrix $\boldsymbol{\Sigma}_{k}$ is positive definite if and only if $r_{k}s_{k}\boldsymbol{I}_{m}-\alpha\boldsymbol{\Lambda}^{2}\succ0$. According to (\ref{eq30ee}), to reach this condition, the value of $r_{k}s_{k}$ should be greater than the maximum singular value of $\nabla \Phi (\tilde{\mathbf{v}}{}^{k})\nabla \Phi (\tilde{\mathbf{v}}{}^{k})^{T}$, i.e., $r_{k}s_{k}>\alpha\rho(\nabla \Phi (\tilde{\mathbf{v}}{}^{k})\nabla \Phi (\tilde{\mathbf{v}}{}^{k})^{T})$. Since $\Vert \nabla \Phi (\tilde{\mathbf{v}}{}^{k})\nabla \Phi (\tilde{\mathbf{v}}{}^{k})^{T}\Vert\ge \rho(\nabla \Phi (\tilde{\mathbf{v}}{}^{k})\nabla \Phi (\tilde{\mathbf{v}}{}^{k})^{T})$ holds, we can take $r_{k}s_{k}>\alpha\Vert \nabla \Phi (\tilde{\mathbf{v}}{}^{k})\nabla \Phi (\tilde{\mathbf{v}}{}^{k})^{T}\Vert$. In the same way, we can derive that $\mathbf{G}_{k}\succ0$ if and only if $r_{k}s_{k}>[\alpha+(1-\alpha)^{2}]\Vert\nabla \boldsymbol{\Phi}(\tilde{\mathbf{v}}^{k})\nabla \boldsymbol{\Phi}(\tilde{\mathbf{v}}^{k})^{T}\Vert.$

\section*{CONFLICT OF INTEREST}
The authors declare that there are no conflicts of interest regarding the publication of this paper.

\biography{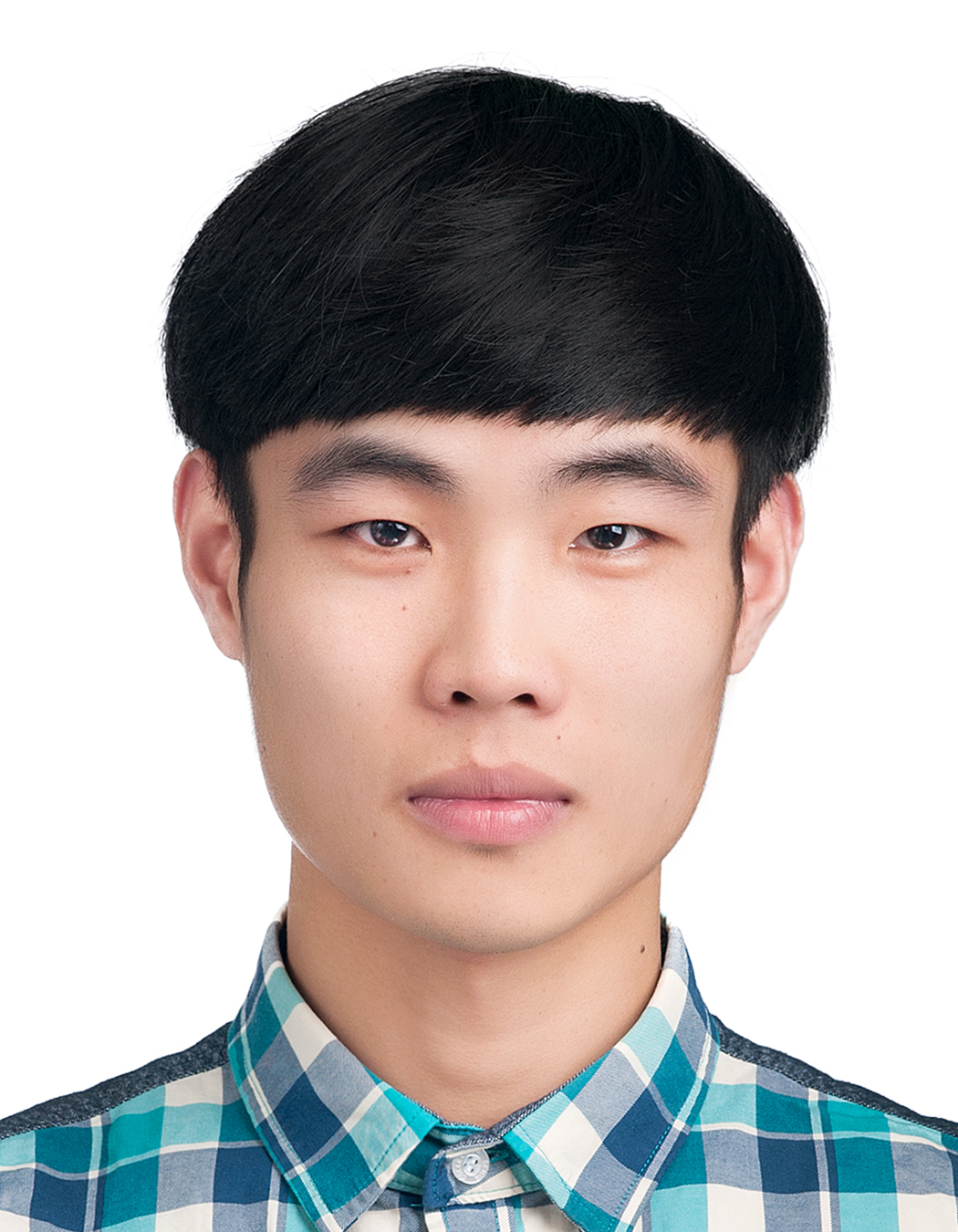}{Sai Wang}{received his M.S. degree in  Information and Telecommunication Engineering from Soongsil University, Seoul, South Korea, in 2019, and a Ph.D. degree in Applied Mathematics from Southern University of Science and Technology (SUSTech), Shenzhen, China, in 2023. He is serving as a postdoctoral researcher in the Department of Electrical and Electronic Engineering at SUSTech. His research interests include nonlinear control, numerical analysis, and machine learning.}

\biography{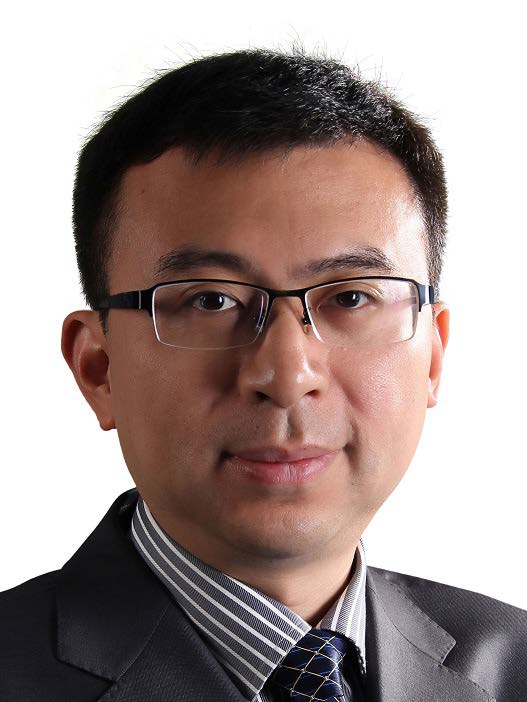}{Yi Gong}{received the B.Eng. and M.Eng. degrees from the Southeast University and the Ph.D. degree from the Hong Kong University of Science and Technology, all in electrical engineering. He was with the Hong Kong Applied Science and Technology Research Institute, Hong Kong, and Nanyang Technological University, Singapore. He is currently a Professor at the Southern University of Science and Technology, Shenzhen, China. His research interests include nonlinear control, mobile computing, and signal processing for wireless communications and related applications. He was on the Editorial Board of the IEEE \textsc{Transactions on Wireless Communications} and the IEEE \textsc{Transactions on Vehicular Technology}.}

\clearafterbiography
\relax 

\end{document}